\documentclass[a4paper,11pt,reqno]{amsart}

\usepackage[style=alphabetic, backend=bibtex, maxbibnames=99]{biblatex}
\AtEveryBibitem{\ifentrytype{article}{\clearfield{url} \clearfield{issn}}{}}

\usepackage[colorlinks=true]{hyperref}
\hypersetup{citecolor=Sepia, linkcolor=Sepia, urlcolor=MidnightBlue}

\usepackage{lmodern}

\usepackage{amsmath}        
\usepackage{amsfonts}       
\usepackage{amsthm}         
\usepackage{graphicx}       
\usepackage[usenames, dvipsnames]{xcolor}  

\usepackage{amssymb}%
\usepackage{mathtools}
\usepackage[shortlabels]{enumitem}
	\setlist[enumerate,1]{label=(\roman*), font=\normalfont}
\usepackage{hyperref}
\usepackage{thmtools}

\theoremstyle{plain}
\newtheorem{theorem}{Theorem}[section]
\newtheorem{lemma}[theorem]{Lemma}
\newtheorem{proposition}[theorem]{Proposition}
\newtheorem{corollary}[theorem]{Corollary}

\theoremstyle{definition}

\newtheorem{remark}[theorem]{Remark}
\newtheorem{example}[theorem]{Example}

\makeatletter
\g@addto@macro\@floatboxreset\centering
\makeatother

\newcommand{\goto}{\rightarrow}

\def\quotient#1#2{%
    \raise1ex\hbox{$#1$}\Big/\lower1ex\hbox{$#2$}%
}

\DeclareMathOperator{\rad}{rad}

\newcommand{\uv}[1]{``{#1}"}

\renewcommand{\H}{\mathbb{H}}

\newcommand{\mE}{\mathcal{E}}
\newcommand{\bX}{\boldsymbol{X}}
\newcommand{\bY}{\boldsymbol{Y}}
\newcommand{\boldzeta}{\boldsymbol{\zeta}}
\newcommand{\boldxi}{\boldsymbol{\xi}}
\newcommand{\boldalpha}{\boldsymbol{\alpha}}

\renewcommand{\b}{\mathfrak{b}}

\newcommand{\simsim}{\stackrel{{\scriptsize{sim}}}{\sim}}

\newcommand{\simstb}{\stackrel{{\scriptsize{stb}}}{\sim}}

\newcommand{\iw}[1]{\mathfrak{i}_{\mathrm{W}}(#1)}
\newcommand{\iql}[1]{\mathfrak{i}_{\mathrm{ql}}(#1)}
\newcommand{\iti}[1]{\mathfrak{i}_{\mathrm{t}}(#1)}

\newcommand{\an}{\mathrm{an}}
\newcommand{\nd}{\mathrm{nd}}

\newcommand{\lc}[1]{\mathrm{lc}(#1)}

\newcommand{\Tg}[2]{\langle D_{#1}^*(#2)\rangle}
\newcommand{\Ng}[2]{\langle D_{#1}^*(#2)^2\rangle}

\newcommand{\ort}{\:\bot\:}
\newcommand{\ql}[1]{\mathrm{ql}(#1)}

\newcommand{\sqf}[1]{\langle #1 \rangle} 
\newcommand{\gensqf}[2]{\langle {#1}_1, \dots, {#1}_{#2} \rangle} 
\newcommand{\nsqf}[1]{\left[{#1} \right]} 
\newcommand{\gennsqf}[3]{\left[{#1}_1,{#2}_1\right]\ort\dots\ort\left[{#1}_{#3},{#2}_{#3}\right]} 

\newcommand{\pf}[1]{\langle\!\langle #1 \rangle\!\rangle} 

\newcommand{\dbrac}[1]{(\!(#1)\!)}

\begin{document}
\title[Quadratic forms over function fields in char. 2]{Isotropy of quadratic forms over function fields in characteristic $2$}
\author{Krist\'yna Zemkov\'a}
\address{Fakult\"at f\"ur Mathematik, Technische Universit\"at Dortmund, D-44221 Dortmund,
Germany}
\address{Department of Mathematics and Statistics, University of Victoria, Victoria BC V8W 2Y2, Canada}
\email{zemk.kr@gmail.com}%
\date{\today}

\begin{abstract}
We extend to characteristic two recent results about isotropy of quadratic forms over function fields. In particular, we provide a characterization of function fields not only of quadratic forms but also more generally of polynomials in several variables, over which a quadratic form becomes isotropic. As an application of these results, we obtain criteria for stable birational equivalence of quadratic forms.
\end{abstract}

\thanks{This work was supported by DFG project HO 4784/2-1. The author further acknowledges support from the Pacific Institute for the Mathematical Sciences and a partial support from the National Science and Engineering Research Council of Canada. The research and findings may not reflect those of these institutions.\\ \indent The author reports there are no competing interests to declare.}
\maketitle


\section{Introduction}

An important question in the algebraic theory of quadratic forms inquires under which conditions an anisotropic quadratic form becomes isotropic after extending scalars to a field extension. In the spotlight are function fields of other quadratic forms, or, more generally, of polynomials. 

In the first case, where the function field is determined by an anisotropic quadratic form, there is an important result called the separation theorem. It was originally proved by Hoffmann \cite{Hof95}, and extended to characteristic two by Hoffmann and Laghribi \cite{HL06}. The theorem says that, for anisotropic quadratic forms $\varphi$ and $\psi$ defined over $F$, the form $\varphi$ stays anisotropic over $F(\psi)$ if  $\dim\varphi\leq 2^n <\dim\psi$ for some $n\geq1$.

Little is known in case of function fields of polynomials but a well-known result in this direction is the so-called norm theorem.
 It was first proved by Knebusch \cite{Kne73}, the full characteristic two version has been proved only recently by Laghribi and Mukhija \cite{LagMuk21}. It characterizes the polynomials $f\in F[X_1,\dots, X_n]$, for which a quadratic form $\varphi$ becomes quasi-hyperbolic over the function field $F(f)$, as the ones for which $f\in G_{F(X_1,\dots,X_n)}(\varphi)$ holds. Note that the result concerns itself with quasi-hyperbolicity, which is a much stronger condition than isotropy.

Regarding isotropy of quadratic forms over the function field of a polynomial, there is another known result which is, in a way, similar to the norm theorem.

\begin{theorem}[{\cite[Th.~18.3]{EKM}}] \label{Th:ExampleOfKnown}
Let $\varphi$ be a quadratic form over $F$, ${f\in F[X]}$ a nonzero polynomial in one variable. Then the following conditions are equivalent:
\begin{enumerate}
	\item $af\in \Tg{F(X)}{\varphi}$ for some $a\in F^*$,
	\item $\varphi_{F(g)}$ is isotropic for each irreducible divisor $g$ occurring to an odd power in the factorization of $f$.
\end{enumerate}
\end{theorem}

\noindent Here, $\Tg{F(X)}{\varphi}$ denotes the multiplicative subgroup of $F(X)^*$ generated by the nonzero elements represented by $\varphi$ over the field $F(X)$, and the theorem holds regardless of the characteristic of the field $F$. 

A generalization of Theorem~\ref{Th:ExampleOfKnown} to several variables and some other connected results have been proved recently by Roussey~\cite{Rous}, but only for fields of characteristic other than two. 

\bigskip

In this article, we extend the work of Roussey to the characteristic two case. After introducing all the necessary definitions and known results in Section~\ref{Sec:Prel}, we generalize in Section~\ref{Sec:Poly} the above-mentioned Theorem~\ref{Th:ExampleOfKnown} to several variables; see Theorem~\ref{Th:EBF}.  In Section~\ref{Sec:Quadr}, we focus on function fields of quadratic forms and give a characterization of isotropy over such fields analogous to the one in Theorem~\ref{Th:ExampleOfKnown}. Here, the more specific assumptions allow us to provide a number of equivalent conditions, some of them seemingly stronger; see Theorem~\ref{Th:CharacterizationIsotropyOverFunField} and Corollary~\ref{Cor:CharacterizationIsotropyOverFunField_Simplified}. Some interesting results arise if we consider the tensor product of a quadratic form with a bilinear Pfister form; this is the topic of the final part of this note, Section~\ref{Sec:Pfister}, and the main result of this section is Theorem~\ref{Th:IsotropyPfisterMultiples}. Moreover, the characterizations of isotropy over a function field of a quadratic form naturally lead to characterizations of stable birational equivalence; see Corollaries~\ref{Cor:CharSTB} and \ref{Cor:CharSTB-PF}.

\section{Preliminaries} \label{Sec:Prel}

Throughout this paper, all fields are of characteristic two.

\subsection{Basic notions}

Let $V$ be a finite dimensional vector space over $F$. We define a \emph{quadratic form} on $V$ as a map $\varphi: V\goto F$ such that
\begin{enumerate}[(1)]
	\item $\varphi(av)=a^2\varphi(v)$ for any $a\in F$ and $v\in V$,
	\item $\b_{\varphi}: V\times V\goto F$, defined by $\b_{\varphi}(v,w)=\varphi(v+w)+\varphi(v)+\varphi(w)$, is a bilinear form.
\end{enumerate}
We define the dimension of $\varphi$ as $\dim V$, denoted $\dim\varphi$. In the following, we will sometimes use \emph{(quadratic) form over $F$} as an abbreviation for quadratic form on a vector space $V$ over $F$.

Any quadratic form $\varphi$ can be written as
\begin{equation}\label{Eq:GenSQF}
\varphi\cong\gennsqf{a}{b}{r}\ort\gensqf{c}{s}
\end{equation}
with $a_i,b_i,c_j\in F$, where $[a,b]$ corresponds to the quadratic form given by the polynomial $aX^2+XY+bY^2$ and $\sqf{c}$ to the one given by $cZ^2$. We call the tuple $(r,s)$ the \emph{type} of the form $\varphi$; it is invariant under isometry. We say that the form $\varphi$ is 
\begin{itemize}
	\item the \emph{zero form} if $r=s=0$,
	\item \emph{nonsingular} if $r>0$ and $s=0$,
	\item \emph{semisingular} if $r>0$ and $s>0$,
	\item \emph{quasilinear} (sometimes also called \emph{totally singular}) if $r=0$ and $s>0$,
	\item and \emph{singular} if $s>0$ (this term covers both semisingular and quasilinear forms).
\end{itemize}
In \eqref{Eq:GenSQF}, the part $\gensqf{c}{s}$ is determined uniquely (up to isometry) by $\rad\b_{\varphi}$; it is called the \emph{quasilinear part} and denoted by $\ql{\varphi}$. 

\bigskip

We say that a quadratic form $\varphi$ is \emph{isotropic} if there exists a nonzero vector $v\in V$ such that $\varphi(v)=0$. Otherwise, the quadratic form is called \emph{anisotropic}. For example, the quadratic form $[0,0]$ is isotropic; it is called \emph{hyperbolic plane} and denoted by $\H$. A second important example of an isotropic quadratic form is $\sqf{0}$. By Witt's decomposition theorem, any quadratic form $\varphi$ can be written as
\begin{equation}\label{Eq:WittDecomp}
\varphi\cong i\times\H\ort\varphi_r\ort\varphi_s\ort j\times\sqf{0}
\end{equation}
with $\varphi_r$ nonsingular, $\varphi_s$ quasilinear and $\varphi_r\ort\varphi_s$ anisotropic. In this decomposition, the form $\varphi_r\ort\varphi_s$ is unique up to isometry. The form $\varphi_s$ is also unique up to isometry, and denoted by $\ql{\varphi}$. The form $\varphi_r$ is generally not unique. Moreover, the numbers $i$ and $j$ are determined uniquely; we define
\begin{itemize}
	\item the \emph{Witt index} of $\varphi$ as $\iw{\varphi}=i$,
	\item the \emph{quasilinear index} (or \emph{defect}) of $\varphi$ as $\iql{\varphi}=j$,
	\item and the \emph{total isotropy index} of $\varphi$ as $\iti{\varphi}=i+j$.
\end{itemize}  
Note that $\iti{\varphi}$ is the dimension of any maximal totally isotropy subspace of $V$. We call the form $\varphi$ \emph{nondefective} if $\iql{\varphi}=0$. Note that for quasilinear forms, being nondefective is equivalent to being anisotropic, whereas all nonsingular forms are nondefective. We define the \emph{nondefective part} of $\varphi$ as $\varphi_{\nd}\cong i\times\H\ort\varphi_r\ort\varphi_s$.

\bigskip

Let $\sigma$ be another quadratic form over $F$. We say that $\sigma$ is \emph{dominated by} $\varphi$, denoted by $\sigma\prec\varphi$, if $\sigma\cong\varphi|_U$ for some vector space $U\subseteq V$. Moreover, we say that $\sigma$ is a \emph{subform} of $\varphi$, denoted by $\sigma\subseteq\varphi$, if there exists a quadratic form $\psi$ such that $\varphi\cong\sigma\ort\psi$. Note that any subform of $\varphi$ is dominated by $\varphi$. If $\sigma$ is nonsigular or both $\varphi$ and $\sigma$ are quasilinear, then the notions of subform and dominance are equivalent, but not so in general: For example, $\sqf{0}\prec\H$, but $\sqf{0}\ort\sqf{a}\not\cong\H$ for any $a\in F$, because isometry preserves the type.


\bigskip

If $E/F$ is a field extension, then we write $\varphi_E$ for the quadratic form on the vector space $V_E=E\otimes_F V$ given by $\varphi_E(a\otimes v)=a^2\varphi(v)$ and with the polar form $\b_{\varphi_E}$ defined by $\b_{\varphi_E}(a\otimes v, b\otimes w)=ab\b_{\varphi}(v,w)$ for any $a,b\in E$ and $v,w\in V$.

\bigskip

Let $f\in F[X_1,\dots, X_n]$ be an irreducible polynomial; we define the \emph{function field} $F(f)$ as the quotient field of the domain $F[X_1,\dots,X_n]/(f)$. 

For a quadratic form $\varphi$ of dimension $n$, the polynomial $\varphi(X_1,\dots,X_n)$  is reducible if and only if the nondefective part $\varphi_{\nd}$ of $\varphi$ is either of the type $(0,1)$, or of the type $(1,0)$ and $\varphi_{\nd}\cong\H$, see~\cite{Ahmad1997}. We say that the quadratic form $\varphi$ is irreducible if the polynomial $\varphi(X_1,\dots,X_n)$ is \emph{irreducible}; in this case, we can define the function field $F(\varphi)$ as above. If $\varphi\cong\sqf{a}$ for some $a\in F$ or $\varphi\cong\H$, then we set $F(\varphi)=F$; here we note that we will prove in Lemma~\ref{Lemma:DefectTransc} that $F(\varphi)\simeq F(\varphi_{\nd})(X_1,\dots,X_j)$ with $j=\iql{\varphi}$, so for example $F(\sqf{a,0,0})\simeq F(X_1,X_2)$.

\subsection{Groups generated by the represented elements} \label{Subsec:GroupsDefs}
For a quadratic form $\varphi$ over $F$, we denote by $D_F(\varphi)$ the set of elements of $F$ represented by $\varphi$ (including zero) and set $D_F^*(\varphi)=D_F(\varphi)\setminus\{0\}$. If $E/F$ is a field extension, then we write $D_E(\varphi)$ (resp. $D_E^*(\varphi)$) instead of $D_E(\varphi_E)$ (resp. $D_E^*(\varphi_E)$). Note that $D_F(\varphi)=D_F(\varphi_{\nd})$ and $D_F^*(\varphi)=D_F^*(\varphi_{\nd})$. 

For $k\geq1$ and quadratic forms $\varphi_1,\dots,\varphi_k$ over $F$, we define
\[D_F(\varphi_1)\cdots D_F(\varphi_k)=\left\{\prod_{i=1}^ka_i~\bigg|~a_i\in D_F(\varphi_i), 1\leq i\leq k\right\}.\]
If $\varphi_1=\dots=\varphi_k=\varphi$, then we write $D_F(\varphi)^k$ for short; analogously for $D_F^*(\varphi_1)\cdots D_F^*(\varphi_k)$ and $D_F^*(\varphi)^k$.
Most importantly, by $\langle D_F^*(\varphi)^k\rangle$ we denote the multiplicative subgroup of $F^*$ generated by $D_F^*(\varphi)^k$, i.e.,
\[\langle D_F^*(\varphi)^k\rangle=\left\{\prod_{i=1}^nb_i~\bigg|~n\geq0, b_1,\dots, b_n\in D_F^*(\varphi)^k\right\}.\]

As a preparation for later proofs, we compare some of the sets and multiplicative groups defined above. 

\begin{lemma} \label{Lemma:MultiplicativeGroups}
Let $\varphi$ be a quadratic form. 
\begin{enumerate}
	\item $\Ng{F}{\varphi}=\Ng{F}{c\varphi}$ for any $c\in F^*$,
	\item $D_F^*(\varphi)^2\subseteq \Ng{F}{\varphi}\subseteq\Tg{F}{\varphi}\subseteq F^*$; moreover,\\ ${F^{*2}\subseteq D_F^*(\varphi)^2}$ if $\varphi\not\cong\sqf{0,\dots,0}$.
	\item $\Ng{F}{\varphi}=\Tg{F}{c\varphi}$ for any $c\in D_F^*(\varphi)$.
\end{enumerate}
\end{lemma}

\begin{proof}
(i) Since $D_F(\varphi)=D_F(c^{-2}\varphi)$, the group $\Ng{F}{c\varphi}$ is generated by the elements $ca\cdot c(c^{-2}b)=ab$ with $a,b\in D_F^*(\varphi)$.

(ii) Let $x\in F^*$, then for any $a\in D_F^*(\varphi)$ we have $a^{-1},ax^2\in D_F^*(\varphi)$; hence, $x^2=a^{-1}\cdot ax^2 \in D_F^*(\varphi)^2$, i.e., $F^{*2}\subseteq D_F^*(\varphi)^2$.  
Let $y\in\Ng{F}{\varphi}$; then $y=\prod_{i=1}^na_ib_i$ for some $n>0$ and $a_i,b_i\in D_F^*(\varphi)$, and hence $y\in\Tg{F}{\varphi}$, i.e., $\Ng{F}{\varphi}\subseteq\Tg{F}{\varphi}$. The remaining inclusions are obvious.

(iii) We have 
\[
\Ng{F}{\varphi}
\stackrel{\text{(i)}}{=}
\Ng{F}{c\varphi} 
\stackrel{\text(ii)}{\subseteq} 
\Tg{F}{c\varphi}
=
\langle cD_{F}^*(\varphi)\rangle
\subseteq
\Ng{F}{\varphi},
\]
where the last inclusion holds, because $c\in D_F^*(\varphi)$.
\end{proof}

\subsection{Quadratic forms over discrete valuation fields}

In this subsection, we denote by $A$ a discrete valuation ring, $(t)$ its maximal ideal with a uniformizing element $t$, $K$ the quotient field of $A$, $v_t$ the valuation on $K$, and $F$ the residue field $A/(t)$. If $\varphi$ is a quadratic form defined over $A$, we denote by $\overline{\varphi}$ its image under the homomorphism $A\goto F$. Moreover, if $\varphi$ is a quadratic form defined over $F$, then we denote by $\varphi_l$ some lifting to $A$, i.e., a quadratic form over $A$ which satisfies $\overline{\varphi_l}\cong\varphi$ (it does not have to be unique). Then $(\varphi_l)_K$ is a form over $K$; by abuse of notation, we will denote this form $\varphi_l$, too.

We call a vector $\boldxi=(\xi_1,\dots,\xi_n)$ with $\xi_i\in A$ \emph{primitive} if $\xi_i\notin(t)$ for some $i$, i.e., if the image $\overline{\boldxi}$ of $\boldxi$ under $A\goto F$ is nonzero. Note that for any nonzero vector $\boldxi'$ over $A$, there exists a scalar multiple of $\boldxi'$ which is primitive: Let $\boldxi'=(\xi'_1,\dots,\xi'_n)$ and set $k=\max\{v_t(\xi'_i)~|~1\leq i\leq n\}$; then the vector $\boldxi=t^{-k}\boldxi'$ is primitive. In particular, if $\varphi$ is an isotropic form over $K$, then there exists a primitive vector $\boldxi$ over $A$ such that $\varphi(\boldxi)=0$.

\begin{lemma} \label{Lemma:CDVfieldsIsotropy}
Let $\varphi$ be a quadratic form over $F$. If $\varphi_l$ is isotropic over $K$, then $\varphi$ is isotropic over $F$.
\end{lemma}

\begin{proof}
Let $\dim\varphi=n$. If $\varphi_l$ is isotropic over $K$, then there exists a primitive vector $\boldxi=(\xi_1,\dots,\xi_n)$ with all $\xi_i\in A$ such that $\varphi_l(\boldxi)=0$. Then $\overline{\boldxi}$ is a nonzero vector over $F$ for which $\overline{\varphi_l}(\overline{\boldxi})=0$. Therefore, $\overline{\varphi_l}\cong\varphi$ is isotropic over $F$.
\end{proof}

\begin{lemma}[{\cite[Lemma~19.5]{EKM}}] \label{Lemma:CDVfieldsAnisotropy}
Let $\varphi_0$, $\varphi_1$ be quadratic forms over $A$. If the forms $\overline{\varphi_0}$, $\overline{\varphi_1}$ are anisotropic over $F$, then $\varphi_0\ort t\varphi_1$ is anisotropic over $K$.
\end{lemma}

\begin{example} \label{Ex:PowSerAnisotropy}
For a field $F$, let $K=F(X)$ (resp. $K=F\dbrac{X}$) be the rational function field (resp. the field of formal power series) in one variable. Then $K$ is a discrete valuation field with respect to the $X$-adic valuation, the valuation ring is $\mathcal{O}_X=\Big\{\frac{f}{g}~\Bigl|~f,g\in F[X], X\nmid g\Big\}$ (resp. $F[\![X]\!]$), and the residue field is $F$.

Let $\varphi_0$, $\varphi_1$ be anisotropic quadratic forms over $F$. By Lemma~\ref{Lemma:CDVfieldsIsotropy}, $\varphi_0$ and $\varphi_1$ are anisotropic over $K$. By Lemma~\ref{Lemma:CDVfieldsAnisotropy}, the quadratic form $\varphi_0\ort X\varphi_1$ is also anisotropic over $K$. Analogously, by considering the anisotropic part, one can show that if $\varphi_0$, $\varphi_1$ are nondefective over $F$, then $\varphi_0\ort X\varphi_1$ is nondefective over $K$.
\end{example}

\begin{lemma} \label{Lemma:CDVfieldsEvenValuation} 
Let $\varphi$ be an anisotropic quadratic form over $F$ and $\dim\varphi=n$. Then $v_t(\varphi_l(\boldxi))$ is even for any nonzero vector $\boldxi=(\xi_1,\dots,\xi_n)\in K^n$. 
\end{lemma}

\begin{proof}
Let $k=\min\{v_t(\xi_i)~|~1\leq i\leq n\}$, and set $\xi_i'= t^{-k}\xi_i$ for each $i$. Then all $\xi'_i\in A$, but at least one of them does not lie in $(t)$. Write $\boldxi'=(\xi_1',\dots,\xi_n')$; then $\boldxi'=t^{-k}\boldxi$, and so $\varphi_l(\boldxi')=t^{-2k}\varphi_l(\boldxi)$. 

Suppose that $v_t(\varphi_l(\boldxi))$ is odd; then $v_t(\varphi_l(\boldxi'))$ is odd, and in particular $v_t(\varphi_l(\boldxi'))\geq1$. Hence, $0=\overline{\varphi_l(\boldxi')}$ in $F$. Since $\overline{\boldxi'}$ is a nonzero vector over $F$, we get that $\overline{\varphi_l}\cong\varphi$ is isotropic over $F$.
\end{proof}

\begin{lemma} \label{Lemma:CDVfieldsEvenValuationConcretely} 
Let $\varphi$ be an anisotropic quadratic form over $F$ of dimension $n$ and assume that $K=F\dbrac{X}$. Let $\boldxi=(\xi_1,\dots,\xi_n)\in K^n$ and $a=\varphi_K(\boldxi)$. Then $v_X(a)=2\min\{v_X(\xi_i)~|~1\leq i\leq n\}$.
\end{lemma}

\begin{proof}
Write 
\[\xi_i=\sum_{j=d_i}^\infty\xi_{ij}X^j \quad \text{and} \quad a=\sum_{j=d_a}^\infty a_jX^j\]
with $\xi_{ij},a_j\in F$ and $\xi_{id_i}, a_{d_a}\neq0$, i.e., $v_X(\xi_i)=d_i$ and $v_X(a)=d_a$. Moreover, set $D=\min\{d_i~|~1\leq i\leq n\}$, and
$\boldalpha=(\alpha_1,\dots,\alpha_n)$ with
\[ 
\alpha_i=\begin{cases}
						\xi_{id_i} & \text{ if } d_i=D,\\
						0 & \text{ otherwise}.
				 \end{cases}
\]
Note that $\varphi_K$ is anisotropic by Lemma~\ref{Lemma:CDVfieldsIsotropy}, because $\varphi_A\cong\varphi_l$ over ${A=F[\![X]\!]}$. Thus, comparing the terms of the lowest degree in $a=\varphi_K(\boldxi)$, we get that $\varphi(X^D\boldalpha)=a_{d_a}X^{d_a}$. Hence, $d_a=2D$.
\end{proof}

\subsection{Known theorems about isotropy of \texorpdfstring{$\varphi_{F(\psi)}$}{varphi over F(psi)}} For the reader's convenience, we provide here the statements which will be needed in our proofs.

\begin{lemma}[{\cite[Prop.~22.9]{EKM}}]\label{Lemma:FunFieldTransc}
Let $\varphi$ be an irreducible nondefective quadratic form. Then the field extension $F(\varphi)/F$ is purely transcendental if and only if $\varphi$ is isotropic.
\end{lemma}

\begin{lemma}\label{Lemma:DefectTransc}
Let $\varphi$ be a quadratic form over $F$. Then the field extension $F(\varphi)/F(\varphi_{\nd})$ is purely transcendental.
\end{lemma}

\begin{proof}
Let $\dim\varphi_{\nd}=n$ and $\iql{\varphi}=j$, i.e., we have $\varphi\cong\varphi_{\nd}\ort j\times\sqf{0}$. Let $\bX=(X_1,\dots,X_n)$ and $\bY=(Y_1,\dots,Y_j)$ be tuples of variables. Then $\varphi(\bX,\bY)=\varphi_{\nd}(\bX)$ as polynomials, and it follows that $F(\varphi)\simeq F(\varphi_{\nd})(\bY)$.
\end{proof}

The following lemma can be also seen as a consequence of Lemma~\ref{Lemma:CDVfieldsIsotropy}.

\begin{lemma}[{\cite[Lemma~7.15]{EKM}}] \label{Lemma:IsotropyTransc}
Let $\varphi$ be a quadratic form over $F$ and $E/F$ a purely transcendental extension. Then $\varphi$ is isotropic over $E$ if and only if $\varphi$ is isotropic over $F$.
\end{lemma}

\begin{lemma}[{\cite[Cor.~3.3]{Lag02}}]\label{Lemma:IsotropySQFoverTSQF}
Let $\varphi$, $\psi$ be anisotropic quadratic forms over $F$. If $\varphi$ is quasilinear and $\psi$ is not quasilinear, then $\varphi_{F(\psi)}$ is anisotropic.
\end{lemma}

\begin{proposition} \label{Prop:TransitionOfIsotropy}
Let $\varphi$, $\psi$, $\sigma$ be quadratic forms over $F$ with $\psi$ nondefective.  If $\varphi_{F(\psi)}$ and $\psi_{F(\sigma)}$ are isotropic, then $\varphi_{F(\sigma)}$ is isotropic as well.
\end{proposition}

\begin{proof}
If all the forms $\varphi$, $\psi$ and $\sigma$ are anisotropic, then the claim coincides with \cite[Prop.~22.16]{EKM}.

If $\varphi$ is isotropic, then $\varphi_{F(\sigma)}$ is also isotropic for trivial reasons. 

If $\psi$ is isotropic and nondefective, then either $\psi\cong\H$ and $F(\psi)\simeq F$, or $F(\psi)/F$ is purely transcendental by Lemma~\ref{Lemma:FunFieldTransc}; anyway, $\varphi$ must be isotropic over $F$, and we are in the previous case.

Now assume that $\sigma$ is isotropic but nondefective. By an analogical argument as above, we obtain that $\psi$ must be isotropic over $F$ (but still nondefective by the assumption). Then again, $\varphi$ is isotropic over $F$ by Lemma~\ref{Lemma:IsotropyTransc}, so $\varphi_{F(\sigma)}$ is isotropic.

Finally, if $\sigma$ is defective, then $F(\sigma)/F(\sigma_\nd)$ is purely transcendental by Lemma~\ref{Lemma:DefectTransc}, and so $\psi_{F(\sigma_\nd)}$ is isotropic by Lemma~\ref{Lemma:IsotropyTransc}. Then $\varphi_{F(\sigma_\nd)}$ is isotropic by the previous part of the proof. Applying Lemma~\ref{Lemma:IsotropyTransc} again, we get that $\varphi_{F(\sigma)}$ is isotropic.
\end{proof}

In the following proposition, $\wp$ denotes the Artin-Schreier map $x\mapsto x^2+x$.

\begin{proposition}[{\cite[Lemma~2.6]{LagMam06} and \cite[Ch.~IV., Th.~4.2]{BaeBOOK}}] \label{Prop:IsotropyOverQuadrExt} \label{Prop:IsotropyOverSepQuadrExt}
Let $\varphi$ be an anisotropic quadratic form over $F$.
\begin{enumerate}
	\item Let $d\in F\setminus F^2$. Then $\varphi_{F(\sqrt{d})}$ is isotropic if and only if there exists $c\in F^*$ such that $c\sqf{1,d}\prec \varphi$.
	\item Let $d\in F\setminus \wp(F)$. Then $\varphi_{F(\wp^{-1}(d))}$ is isotropic if and only if there exists $c\in F^*$ such that $c\nsqf{1,d}\subseteq \varphi$.
\end{enumerate}
\end{proposition}

\section{Function fields of polynomials} \label{Sec:Poly}

In this section, we look at function fields of polynomials. In particular, we prove a generalization of Theorem~\ref{Th:ExampleOfKnown} to more variables. All statements have their counterparts in characteristic other than two. Sometimes, the transition of the proof to characteristic two is straightforward, but often some obstacles specific to characteristic two appear.

Throughout this section, we denote $\bX=(X_1,\dots,X_l)$ with $l>0$. For a polynomial $f\in F[\bX]$, we denote by $\deg_{X_i}{f}$ the maximal degree of the variable $X_i$ appearing in $f$, and by $\deg f$ the total degree
of the leading term of $f$ with respect to the lexicographical ordering. Furthermore, $\lc{f}$ denotes the coefficient of the leading term of $f$, and we say that $f$ is \emph{monic} if $\lc{f}=1$. 

\begin{lemma} \label{Lemma:RousAn_SuffCond}
Let $\varphi$ be a quadratic form over $F$. Let $f\in F[\bX]$ be an irreducible polynomial. If there exists $a\in F^*$ such that $af\in\Tg{F(\bX)}{\varphi}$, then $\varphi_{F(f)}$ is isotropic.
\end{lemma}

\begin{proof}
Since the function fields $F(f)$ and $F(af)$ coincide, we can assume without loss of generality $a=1$. Set $n=\dim\varphi$. Let
\[f=\prod_{i=1}^k\varphi(\boldxi'_i)\]
for some $k>0$ and $\boldxi'_i=(\xi'_{i1},\dots,\xi'_{in})$ with $\xi'_{ij}\in F(\bX)$ for all $i,j$. For each $i$, we can find $h_i\in F[\bX]$ such that $h_i\xi'_{ij}\in F[\bX]$ for all $j$; we denote $\xi_{ij}=h_i\xi'_{ij}$ and $\boldxi_i=(\xi_{i1},\dots,\xi_{in})$. Moreover, set $h=\prod_{i=1}^kh_i$. Then
\[h^{2}f=\prod_{i=1}^k\varphi(\boldxi_i).\]

If there exists an $i\in\{1,\dots,k\}$ such that $f\mid \xi_{ij}$ for all $j$, then $f^2\mid h^{2}f$; since $f$ is irreducible, we get $f\mid h$. Hence, we can replace $\xi_{ij}$ by $\frac{\xi_{ij}}{f}$ for each $1\leq j\leq n$, and $h$ by $\frac{h}{f}$. Repeating this step if necessary, we may assume that for each $i\in\{1,\dots,k\}$ there exists at least one $j\in\{1,\dots,n\}$ such that $f\nmid \xi_{ij}$.

Since $f\mid\prod_{i=1}^k\varphi(\boldxi_i)$ and $f$ is irreducible, there exists an $i\in\{1,\dots,k\}$ such that $f\mid\varphi(\boldxi_i)$, i.e., $\varphi(\boldxi_i)=fg$ for some $g\in F[\bX]$. Then $\varphi(\boldxi_i)=0$ over the domain $F[\bX]/(f)$, and hence also over its quotient field $F(f)$. By the previous paragraph, $\boldxi_i$ is a nonzero vector over $F(f)$. Therefore, $\varphi_{F(f)}$ is isotropic.
\end{proof}

\begin{remark}
Note that in the previous lemma, the quadratic form $\varphi$ could be exchanged for a homogeneous polynomial of any degree $d>0$, we would only have to rewrite the squares in the proof by $d$-th powers. In particular, the lemma remains true for so-called quasilinear $p$-forms defined over a field of characteristic $p$.
\end{remark}

\begin{lemma} \label{Lemma:RepreOfLeadingCoeff}
Let $\varphi$ be a quadratic form over $F$. If $f\in F[\bX]$ is such that $f\in D_{F(\bX)}^*(\varphi)^m$ for some $m>0$, then $\lc{f}\in D_F^*(\varphi)^m$.
\end{lemma}

\begin{proof}
First, note that $(\varphi_{\nd})_{F(\bX)}\cong(\varphi_{F(\bX)})_{\nd}$ by Lemma~\ref{Lemma:CDVfieldsIsotropy}, and that ${D_{F(X)}^*(\varphi_{\nd})=D_{F(X)}^*(\varphi)}$ and $D_{F}^*(\varphi_{\nd})=D_{F}^*(\varphi)$; therefore, we can assume without loss of generality that $\varphi$ is nondefective. Second, if $\H\subseteq\varphi$, then $D_F^*(\varphi)^m=F^*$, and the claim is trivial. Therefore, suppose that $\varphi$ is anisotropic.

Let $n=\dim\varphi$. Using the same construction as in the proof of Lemma~\ref{Lemma:RousAn_SuffCond}, we find $h$ and $\boldxi_i=(\xi_{i1},\dots,\xi_{in})$ with $h,\xi_{ij}\in F[\bX]$ such that
\[fh^2=\prod_{i=1}^m\varphi(\boldxi_i).\] 
Note that we can assume $h$ to be monic. 

For each $i$, we set $d_i=\max\{\deg\xi_{ij}~|~1\leq j\leq n\}$, and
\[\alpha_{ij}=\begin{cases} \lc{\xi_{ij}} & \text{if } \deg\xi_{ij}=d_i, \\ 0 & \text{otherwise.}\end{cases}\]
Writing $\boldalpha_i=(\alpha_{i1},\dots,\alpha_{in})$ and recalling that $\varphi$ is anisotropic, it follows that $\varphi(\boldalpha_i)$ is the leading coefficient of $\varphi(\boldxi_i)$. Therefore, $\prod_{i=1}^m\varphi(\boldalpha_i)$ is the leading coefficient of $\prod_{i=1}^m\varphi(\boldxi_i)=fh^2$. Since $h$ is monic, we have $\lc{f}=\lc{fh^2}$, and the claim follows.
\end{proof}

\begin{lemma}
\label{Lemma:DivAndReprPolyOneVar}
Let $f\in F[X]$ be a monic irreducible polynomial in one variable and $\varphi$ a nondefective quadratic form over $F$ of dimension $n$. Suppose that $f$ divides $\varphi(\xi_1,\dots,\xi_n)$ for some $\xi_i\in F[X]$, $1\leq i\leq n$, not all of them divisible by $f$. Then $f\in D_{F(X)}^*(\varphi)^m$ for some $m\leq\deg f$.
\end{lemma}

\begin{proof}
First assume that $\dim\varphi=1$; then $\varphi\cong\sqf{a}$ for some $a\in F^*$, and $f\mid\varphi(\xi_1)=a\xi_1^2$, i.e., $f$ divides $\xi_1$, a contradiction to our assumption. Hence, this case cannot happen.

If $\varphi$ is isotropic, then, since it is nondefective by the assumption, we have $\H\subseteq\varphi$. Hence, $D_{F(X)}^*(\varphi)=F(X)^*$, and $f\in D_{F(X)}^*(\varphi)$ trivially.

From now on, suppose that $\varphi$ is anisotropic and $\dim\varphi\geq2$. If $\deg f=1$, then $f(X)=X-b$ for some $b\in F$. It means that $\varphi(\xi_1(X),\dots,\xi_n(X))=(X-b)g(X)$ for some $g\in F[X]$, and hence $\varphi(\xi_1(b),\dots,\xi_n(b))=0$. Since $\varphi$ is anisotropic, we get $\xi_i(b)=0$ for every $1\leq i \leq n$; but in that case $f(X)=X-b$ divides all $\xi_i$'s, a contradiction.

Assume $\deg f=2$; then, up to a linear substitution, there are only two possibilities: either $f(X)=X^2+c$ or $f(X)=X^2+X+c$ for some $c\in F^*$. Pick a root $\gamma\in\overline{F}$ of $f$ (i.e., $\gamma^2=c$, resp. $\gamma^2+\gamma=c$). Since we can find $g(X)\in F[X]$ such that $f(X)g(X)=\varphi(\xi_1(X),\dots,\xi_n(X))$, it follows that over $F(\gamma)$, we have $\varphi(\xi_1(\gamma),\dots,\xi_n(\gamma))=0$. Note that if $\xi_i(\gamma)=0$ for all $1\leq i\leq n$, then $X+\gamma$ divides each $\xi_i$ over $F(\gamma)$; as $f$ is the minimal polynomial of $\gamma$ over $F$, it implies that $f$ divides each $\xi_i$ over $F$, which is a contradiction. Therefore, not all $\xi_i(\gamma)=0$, and so $\varphi_{F(\gamma)}$ is isotropic. Let us apply Proposition~\ref{Prop:IsotropyOverSepQuadrExt}: We can find some $c'\in F^*$ such that in the case of $f(X)=X^2+c$, we have $c'\sqf{1,c}\prec\varphi$, and in the case of $f(X)=X^2+X+c$, we have $c'\nsqf{1,c}\subseteq\varphi$. In both cases we get $\frac{1}{c'}\in D_F^*(\varphi)$ and $c'f(X)\in D_{F(X)}^*(\varphi)$; thus, $f(X)\in  D_{F}^*(\varphi) D_{F(X)}^*(\varphi)\subseteq D_{F(X)}^*(\varphi)^2$.

Now we will assume $\deg f\geq2$ and proceed by induction on $\deg f$. For each $i$, we divide $\xi_i$ by $f$ with a remainder $\xi_i'$:
\[\xi_i=f\zeta_i+\xi_i' \quad \text{ for some } \zeta_i, \xi_i'\in F[X], \ \deg \xi'_i<\deg f.\]
We denote
\[\boldzeta=(\zeta_1,\dots,\zeta_n), \quad \boldxi=(\xi_1,\dots,\xi_n), \quad \boldxi'=(\xi_1',\dots,\xi_n');\]
then $\boldxi=f\boldzeta+\boldxi'$, and by the assumption there exists $g\in F[X]$ such that
\[fg=\varphi(\boldxi)=\varphi(\boldxi')+f^2\varphi(\boldzeta)+f\b_\varphi(\boldzeta,\boldxi').\]
Let $g+f\varphi(\boldzeta)+\b_\varphi(\boldzeta,\boldxi')=ah$ with $a\in F^*$ and $h\in F[X]$ monic; then $\varphi(\boldxi')=afh$, and we have
\[\deg f+\deg h=\deg\varphi(\boldxi')\leq 2\max\{\deg \xi'_i~|~i=1,\dots,n\}<2\deg f,\]
i.e., $\deg h <\deg f$. Moreover, since there is at least one $\xi_i$ non-divisible by $f$, we have $\boldxi'\neq0$. As $\varphi$ is anisotropic, it follows that $\varphi(\boldxi')\neq0$; hence, $\deg\varphi(\boldxi')$ is even, so $\deg f$ and $\deg h$ have necessarily the same parity. Therefore, $\deg h\leq\deg f-2$.

Write $h=\prod_{j=1}^rh_j$, where $h_j\in F[X]$ is a monic irreducible polynomial for each $j$. If there is a $k\in\{1,\dots,r\}$ such that $h_k\mid \xi'_i$ for every $i$, then $h_k^2\mid \varphi(\boldxi')=afh$. Since both $h_k$ and $f$ are irreducible and $\deg h_k<\deg f$, it follows that $h_k^2\mid h$. In that case we replace $\boldxi'$ by $\left(\frac{\xi'_1}{h_k},\dots,\frac{\xi'_n}{h_k}\right)$ and $h$ by $\frac{h}{h_k^2}$. Repeating this process if necessary, we end up with $afh''=\varphi(\boldxi'')$, $\boldxi''=(\xi_1'',\dots,\xi_n'')\in F[X]^n$ and $h''=\prod_{j=1}^{s}h_j$, where, for each $1\leq j\leq s$, the polynomial $h_j$ is monic irreducible and does not divide all $\xi_i''$'s.

Since both $f$ and $h''$ are monic, $a$ is the leading coefficient of $afh''=\varphi(\boldxi'')$. Therefore, $a\in D_F^*(\varphi)$ by Lemma~\ref{Lemma:RepreOfLeadingCoeff}.

If $h''=1$, then
\[f=\frac{1}{a}\varphi(\boldxi'')\in D_F^*(\varphi)D_{F(X)}^*(\varphi)\subseteq D_{F(X)}^*(\varphi)^2,\]
and we are done since we assumed $\deg f\geq 2$. If $h''\neq1$, then we know for each $1\leq j\leq s$ that $\deg h_j<\deg f$, and $h_j$ is a monic irreducible polynomial dividing $\varphi(\boldxi'')$ but not dividing all $\xi_i''$. Hence, by the induction hypothesis, $h_j\in D_{F(X)}^*(\varphi)^{m_j}$ for some $m_j\leq\deg h_j$. It also follows that $\frac{1}{h_j}\in D_{F(X)}^*(\varphi)^{m_j}$. Consequently,
\[\frac{1}{h''}=\prod_{j=1}^s\frac{1}{h_j}\in D_{F(X)}^*(\varphi)^m \quad \text{ where } m=\sum_{j=1}^sm_j,\]
and thus
\[f=\frac{1}{ah''}\varphi(\boldxi'')\in D_F^*(\varphi) D_{F(X)}^*(\varphi)^mD_{F(X)}^*(\varphi)\subseteq D_{F(X)}^*(\varphi)^{m+2},\]
where 
\[m+2\leq\deg h''+2\leq\deg h+2\leq \deg f. \qedhere\] 
\end{proof}

With Lemmas~\ref{Lemma:RousAn_SuffCond} and \ref{Lemma:DivAndReprPolyOneVar} at hand, we can prove our first characterization of the isotropy of $\varphi_{F(f)}$, one very similar to Theorem~\ref{Th:ExampleOfKnown}. Unlike in that theorem, we assume $f$ to be irreducible. On the other hand, we give a bound on the power of $D_{F(X)}^*(\varphi)$ necessary to represent $f$.

\begin{proposition}
\label{Prop:RousAn_OneVar}
Let $\varphi$ be a nondefective quadratic form over $F$ such that $1\in D_F^*(\varphi)$, and let $f\in F[X]$ be a monic irreducible polynomial in one variable. Then the following are equivalent:
\begin{enumerate}
	\item $f\in D_{F(X)}^*(\varphi)^m$, $m\leq\deg f$;
	\item $\varphi_{F(f)}$ is isotropic.
\end{enumerate}
\end{proposition}

\begin{proof}
The implication (i) $\Rightarrow$ (ii) is covered by Lemma~\ref{Lemma:RousAn_SuffCond}. To prove the converse, assume that $\varphi_{F(f)}$ is isotropic and denote $n=\dim\varphi$. Then we can find $\overline\xi_1, \dots,\overline\xi_n\in F(f)=F[X]/(f)$, not all zero, such that $\varphi_{F(f)}(\overline\xi_1, \dots,\overline\xi_n)=0$. For each $1\leq i\leq n$, let $\xi_i\in F[X]$ be such that the image of $\xi_i$  in $F[X]/(f)$ is precisely $\overline\xi_i$. Then $\varphi(\xi_1,\dots,\xi_n)=fh$ for some $h\in F[X]$. Note that since not all of the $\overline\xi_i$ were zero, not all of the $\xi_i$ are divisible by $f$. Hence, the claim follows by Lemma~\ref{Lemma:DivAndReprPolyOneVar}.
\end{proof}

Now we extend the previous proposition to polynomials in more variables. The idea of the proof is to treat the polynomial $f\in F[X_1,\dots,X_l]$ as a polynomial in variable $X_1$ over the field $F(X_2,\dots,X_l)$. However, when viewed as an element of $F(X_2,\dots,X_l)[X_1]$, the polynomial $f$ may not be monic anymore; denote its leading coefficient by $\alpha$. We can apply Proposition~\ref{Prop:RousAn_OneVar} on $\frac1\alpha f$. The most difficult part of the proof is then showing that $\alpha\in D^*_{F(X_2,\dots,X_l)}(\varphi)^k$ for some $k$.

\begin{theorem}
\label{Th:IsotropyOverPolyFunField}
Let $\varphi$ be a nondefective quadratic form over $F$ such that $1\in D_F^*(\varphi)$. Let $f\in F[\bX]$ be a monic irreducible polynomial such that $\deg f\geq1$. Then the following are equivalent:
\begin{enumerate}
	\item $f\in D_{F(\bX)}^*(\varphi)^m$ with $m\leq \deg f$,
	\item $\varphi_{F(f)}$ is isotropic.
\end{enumerate}
\end{theorem}

\begin{proof}
As $D_{F(\bX)}^*(\varphi)^m\subseteq\Tg{F(\bX)}{\varphi}$, the implication (i) $\Rightarrow$ (ii) is covered by Lemma~\ref{Lemma:RousAn_SuffCond}. Let us prove the converse. If $\varphi$ is isotropic, then necessarily $\H\subseteq \varphi$, and $f\in D_{F(\bX)}^*(\varphi)$ for trivial reasons. Thus, assume that $\varphi$ is anisotropic.

Note that we can assume $\deg_{X_i}f>0$ for all $1\leq i\leq l$; otherwise, we pick the maximal subset $\{i_1,\dots,i_{\tilde{l}}\}$ of $\{1,\dots,l\}$ such that $\deg_{X_{i_j}}f>0$ for all $1\leq j\leq \tilde{l}$ (where $\tilde{l}\geq1$ by the assumption), and for $\widetilde\bX=(X_{i_1},\dots,X_{i_{\tilde{l}}})$ prove that $f\in D_{F(\widetilde\bX)}^*(\varphi)^m$ for some $m$. Since $D_{F(\widetilde\bX)}^*(\varphi)^m\subseteq D_{F(\bX)}^*(\varphi)^m$, the claim follows.

\bigskip

We proceed by induction on the number of variables $l$. For $l=1$, we apply Proposition~\ref{Prop:RousAn_OneVar}. 

Assume $l\geq2$ and denote $\bX'=(X_2,\dots,X_l)$, $d=\deg_{X_1}f$ and $n=\dim\varphi$. Then we can find polynomials $f_d\in F[\bX']$ and $\widetilde{f}\in F[\bX]$ such that 
\[f=f_dX_1^d+\widetilde{f} \] 
and $\deg_{X_1}\widetilde{f}<d$; then $\deg f_d\leq\deg f-d$. Note that, since $f$ is monic, $f_d$ is monic as well. Furthermore, consider $g=\frac{f}{f_d}$ as an element of $F(\bX')[X_1]$; then $\deg_{X_1} g=d$ and by Gauss' lemma, $g$ is a monic irreducible polynomial in $X_1$ over $F(\bX')$. Since $F(\bX')(g)=F(f)$, the quadratic form $\varphi$ is isotropic over $F(\bX')(g)$, and, by Proposition~\ref{Prop:RousAn_OneVar}, we get 
\[g\in D_{F(\bX')(X_1)}^*(\varphi)^{d'}=D_{F(\bX)}^*(\varphi)^{d'}\]
for some $d'\leq d$. Since $1\in D_F^*(\varphi)$, we can assume $d'=d$.

Now let $h\in F[\bX]$ be such that $gh^2\in D_{F[\bX]}^*(\varphi)^d$, i.e.,
\[gh^2=\varphi(\xi_{11},\dots,\xi_{1n})\cdots\varphi(\xi_{d1},\dots,\xi_{dn})\]
for some $\xi_{ij}\in F[\bX]$; this can be rewritten as
\begin{equation}
fh^2=f_d\,\varphi(\xi_{11},\dots,\xi_{1n})\cdots\varphi(\xi_{d1},\dots,\xi_{dn}).
\label{Eq:fh2fdvarphi}
\end{equation}
Moreover, for each $i$, we can assume that the polynomials $\xi_{i1},\dots,\xi_{in}$ have no common divisor in $F[\bX]$:  If $q\mid \xi_{ij}$ for all $j$, then necessarily $q^2\mid fh^2$, and hence $q\mid h$ since $f$ is irreducible. In that case, we can consider $\frac{\xi_{i1}}{q},\dots,\frac{\xi_{in}}{q}$ and $\frac{h}{q}$ instead.

Write $f_d=r^2s$ with $r,s\in F[\bX']$ monic polynomials (recall that $f_d$ is monic) such that $s$ has no square factors. If $s=1$, then clearly $f_d=r^2\in D_{F(\bX')}^*(\varphi)$. If $s\neq 1$, then proving $s\in D_{F(\bX')}^*(\varphi)^{m'}$ for some $m'>0$ will imply that also $f_d=r^2s\in D_{F(\bX')}^*(\varphi)^{m'}$.

\smallskip

Suppose $s\neq 1$. For each monic irreducible polynomial $t\in F[\bX']$ such that $t\mid s$, we proceed as follows: First, note that $t\nmid f$, because $f$ is irreducible, and by the assumption $f\notin F[\bX']$. Let us compare the $t$-adic valuation $v_t$ on the left and right hand side of \eqref{Eq:fh2fdvarphi}: Since $v_t(f)=0$, the value $v_t(fh^2)$ must be even. On the other hand, $v_t(f_d)$ is odd by the assumption, and hence there exists a $k\in\{1,\dots,d\}$ such that $v_t(\varphi(\xi_{k1},\dots,\xi_{kn}))$ is odd. Analogously as in the proof of Lemma~\ref{Lemma:CDVfieldsEvenValuation}, it follows that $\varphi$ is isotropic over $F[\bX]/(t)$, and hence also over the quotient field $\mathrm{Quot}(F[\bX]/(t))$. As $t\in F[\bX']$, we have $F[\bX]/(t)\simeq F([\bX']/(t))[X_1]$, and so $\mathrm{Quot}(F[\bX]/(t))\simeq F(t)(X_1)$. Therefore, $\varphi$ is isotropic over $F(t)(X_1)$,  and, by Lemma~\ref{Lemma:IsotropyTransc}, $\varphi$ is also isotropic over $F(t)$. By the induction hypothesis, $t\in D_{F(\bX')}^*(\varphi)^{\deg t}$.

\smallskip

All in all, we get 
\[f_d= r^2\prod_{t\mid s \text{ irred.}}\!\!\! t\ \in D_{F(\bX')}^*(\varphi)^{m'}\] 
where
\[{m'}=\begin{cases} 1& \text{ if } s=1, \\ \deg s & \text{ otherwise},\end{cases}\]
and we have $m'\leq\deg f_d\leq\deg f-d$. Hence,
\[f=f_d\,g\in D_{F(\bX)}^*(\varphi)^m\]
with $m=m'+d\leq\deg f$ as claimed.
\end{proof}

As the final step in reaching the goal of this section, we consider reducible polynomials in more variables. We obtain a characteristic two version of \cite[Th.~1]{BF95}.

\begin{theorem} 
\label{Th:EBF}
Let $\varphi$ be a nondefective quadratic form over $F$ such that $1\in D_F(\varphi)$. Let $a\in F^*$ and $f_1,\dots,f_r, g\in F[\bX]$ be monic polynomials with $f_1,\dots,f_r$ distinct and irreducible. Suppose that $f=af_1\cdots f_rg^2$. Then the following are equivalent:
\begin{enumerate}
	\item $f\in \Tg{F(\bX)}{\varphi}$,
	\item $a\in \Tg{F}{\varphi}$ and $f_k\in \Tg{F(\bX)}{\varphi}$ for every $1\leq k \leq r$,
	\item $a\in \Tg{F}{\varphi}$ and  $\varphi_{F(f_k)}$ is isotropic for every $1\leq k\leq r$.
\end{enumerate}
\end{theorem}

\begin{proof}
Applying Theorem~\ref{Th:IsotropyOverPolyFunField} to each $f_k$, $1\leq k\leq r$, we get the implication (iii) $\Rightarrow$ (ii). As $g^2\in D_{F(\bX)}^*(\varphi)$, the implication (ii) $\Rightarrow$ (i) is trivial. Hence, we only need to prove (i) $\Rightarrow$ (iii).

Suppose $f\in\Tg{F(\bX)}{\varphi}$. By Lemma~\ref{Lemma:RepreOfLeadingCoeff}, $a=\lc{f}\in\Tg{F}{\varphi}$. By clearing denominators, we may assume 
\[\prod_{i=1}^m\varphi(\boldxi_i)=fh^2=af_1\cdots f_rg^2h^2\]
for some $m>0$, $\boldxi_i=(\xi_{i1},\dots,\xi_{in})\in F[\bX]^n$ and a monic polynomial $h\in F[\bX]$. Let $\gamma_i=\gcd(\xi_{i1},\dots,\xi_{in})$. Without loss of generality, we can assume $\gamma_i=1$ for all $1\leq i\leq m$; otherwise we replace $\xi_{ij}$ by $\frac{\xi_{ij}}{\gamma_i}$ and $gh$ by $\frac{gh}{\gamma_i}$. 
For any $1\leq k\leq r$, we know that $f_k$ is irreducible, and hence $f_k$ divides $\varphi(\boldxi_i)$ for some $i$; thus $\varphi(\boldxi_i)=0$ over $F(f_k)$. By the assumption above, $\boldxi_i$ is a nonzero vector over $F[\bX]/(f_k)$, and hence also over $F(f_k)$; it follows that $\varphi_{F(f_k)}$ is isotropic.
\end{proof}

\begin{remark}
Note that by the transition to more variables, we lost the upper bound for the necessary power of $D_{F(\bX)}^*(\varphi)$. But if we consider in Theorem~\ref{Th:EBF} only monic polynomials, then we can also keep the upper bounds. In particular, with the notation from the theorem, the following are equivalent for monic $f$ (i.e., $a=1$):
\begin{enumerate}
	\item $f\in D_{F(\bX)}^*(\varphi)^m$ with $m\leq \deg f$,
	\item $f_k\in D_{F(\bX)}^*(\varphi)^{m_k}$ with $m_k\leq \deg f_k$ for every $1\leq k \leq r$,
	\item $\varphi_{F(f_k)}$ is isotropic for every $1\leq k\leq r$.
\end{enumerate}
\end{remark}

\section{Function fields of quadratic forms} \label{Sec:Quadr}

In this section, we concentrate on the isotropy over a function field of a quadratic form. Most of the proofs in this section must deal with the obstacle of \uv{two types} of isotropy: Unlike in the case of characteristic other than two, the isotropy of a quadratic form $\varphi$ does not necessarily mean $\H\subseteq\varphi$, because $\varphi$ may be just defective. So an isotropic form does not have to be universal, which makes both the assertions and the proofs more elaborate.

For a start, we need to prepare some lemmas. In particular, the first lemma basically says that discarding the defect does not affect the Witt index.

\begin{lemma} \label{Lemma:TechnicalIsotropyOverFunField}
Let $\varphi$ be a semisingular and $\psi$ an arbitrary quadratic form over $F$, and let $E/F$ be a field extension. Assume that $\iw{\varphi_{F(\psi)}}>0$ and $\iw{\varphi_E}=0$. Let $\varphi'$ be a quadratic form over $E$ such that $\varphi'\cong(\varphi_E)_{\an}$. Then $\iw{\varphi'_{E(\psi)}}>0$.
\end{lemma}

\begin{proof}
As $\iw{\varphi_E}=0$, we have $\varphi'\cong(\varphi_E)_\nd$; hence, 
\[\varphi_{E(\psi)}\cong\varphi'_{E(\psi)}\ort\iql{\varphi_E}\times\sqf{0},\] 
and it follows that $\iw{\varphi'_{E(\psi)}}=\iw{\varphi_{E(\psi)}}$. Therefore, $\iw{\varphi'_{E(\psi)}}>0$.
\end{proof}

\begin{lemma} \label{Lemma:MainPart-FormsAnisOverE}
Let $\varphi,\psi$ be nondefective quadratic forms over $F$ such that $\dim\varphi, \dim\psi\geq2$, and let $\varphi_{F(\psi)}$ be isotropic. If $E/F$ is a field extension such that both $\varphi_E$ and $\psi_E$ are anisotropic, then $D_E^*(\psi)^2\subseteq D_E^*(\varphi)^2$.
\end{lemma}

\begin{proof}
Let $c\in D_E^*(\psi)$; then  $1\in D_E^*(c\psi)$. Let $a\in D_E^*(c\psi)$. It suffices to show that $a\in D_E^*(\varphi)^2$.  

First, assume $a=x^2$ for some $x\in E^*$; then $a=x^2=bx^2\cdot\frac{b}{b^2}\in D_E^*(\varphi)^2$ for any $b\in D_E^*(\varphi)$. Now suppose $a\notin E^2$. Let $V_{\psi_E}$ be the underlying vector space of $\psi_E$, and let $u,v\in V_{\psi_E}$ be such that $c\psi_E(u)=1$ and $c\psi_E(v)=a$. Note that $u,v$ are linearly independent: if $v=tu$ for some $t\in E^*$, then $a=c\psi_E(v)=c\psi_E(tu)=t^2\in E^2$, a contradiction.

If $\b_{\psi_E}(u,v)=0$, then $\sqf{1,a}\prec c\psi_E$, and hence $\psi_{E(\sqrt{a})}$ is isotropic. By Proposition~\ref{Prop:TransitionOfIsotropy}
, $\varphi_{E(\sqrt{a})}$ is also isotropic, and by Proposition~\ref{Prop:IsotropyOverQuadrExt} we have $b\sqf{1,a}\prec\varphi_E$ for some $b\in E^*$. In particular, $b,ba\in D_E^*(\varphi)$, and hence $a=ba\cdot\frac{b}{b^2}\in D_E^*(\varphi)^2$.

Let $\b_{\psi_E}(u,v)\neq0$. Set $s=\b_{\psi_E}(u,v)$; then $\nsqf{1,s^{-2}a}\subseteq c\psi_E$. As $\psi_E$ is anisotropic by the assumption, it follows that $s^{-2}a\notin\wp(E)$. So, considering $\alpha=\wp^{-1}(s^{-2}a)\in\overline{F}$ (where $\overline{F}$ denotes an algebraic closure of $F$), we get $[E(\alpha):E]=2$ and $\psi_{E(\alpha)}$ is isotropic. Similarly as above, it follows that $\varphi_{E(\alpha)}$ is isotropic by Proposition~\ref{Prop:TransitionOfIsotropy}. Invoking Proposition~\ref{Prop:IsotropyOverSepQuadrExt}, we get $b\nsqf{1,s^{-2}a}\subseteq\varphi_E$ for some $b\in D_E^*(\varphi)$. In particular, $b,bs^{-2}a\in D_E^*(\varphi)$, hence $a=\frac{ab}{s^2}\cdot\frac{bs^2}{b^2}\in D_E^*(\varphi)^2$.
\end{proof}

\begin{proposition}\label{Prop:RousAn_NecCond}
Let $\varphi,\psi$ be nondefective quadratic forms over $F$. Set
\[
\mE=\begin{cases}
\{E~|~E/F \text{ an extension s.t. } \iql{\varphi_E}=0\} & \text{ if } \psi \text{ quasilinear,}\\
\{E~|~E/F \text{ an extension}\} & \text{ otherwise.}
\end{cases}
\]
If $\varphi_{F(\psi)}$ is isotropic, then $D_E^*(\psi)^2\subseteq D_E^*(\varphi)^2$ for every $E\in\mE$.
\end{proposition}

\begin{proof}
 We will prove the proposition in several steps, starting with some trivial cases.

\smallskip

(0) Let $E_0/F$ be a field extension such that $\iw{\varphi_{E_0}}>0$, i.e., $\H_{E_0}\subseteq\varphi_{E_0}$. Then $D_{E_0}^*(\varphi)=E_0^*$, and trivially, $D_{E_0}^*(\psi)^2\subseteq D_{E_0}^*(\varphi)^2$.

\smallskip 

(1) Since a nondefective one-dimensional quadratic form cannot become isotropic, we may assume $\dim\varphi\geq2$.  

If $\dim\psi=1$, then $F(\psi)=F$. If $\psi$ is isotropic (but nondefective by the assumption), then the field extension $F(\psi)/F$ is purely transcendental by Lemma~\ref{Lemma:FunFieldTransc}. In both cases we get that  $\varphi$ is isotropic over $F$ (in the latter case by Lemma~\ref{Lemma:IsotropyTransc}). 

If $\varphi$ is isotropic over $F$, then (since it is nondefective), $\H\subseteq \varphi$. Then also $\H_E\subseteq \varphi_E$ for any field extension $E/F$, and we are done by (0) with $E_0=E$.

\smallskip 

(2) Suppose that $\varphi$, $\psi$ are anisotropic, and let $E\in\mE$. We claim that there exists a form $\psi'\subseteq\psi$ (over $F$) such that $\psi'_E\cong(\psi_E)_\nd$: Write $\psi\cong\psi_r\ort\ql{\psi}$. Then, by \cite[Lemma~2.2]{HL04}, there exists an anisotropic form $\sigma\subseteq\ql{\psi}$ over $F$ such that $\sigma_E\cong(\ql{\psi}_E)_\an$. Therefore, 
\[(\psi_r\ort\sigma)_E\cong (\psi_r)_E\ort(\ql{\psi}_E)_\an\cong(\psi_E)_\nd.\] 
Setting $\psi'=\psi_r\ort\sigma$, the claim follows. Furthermore, note that we have $D_E^*(\psi)=D_E^*(\psi')$. 

If $\dim\psi'=1$, then $D_E^*(\psi')^2=E^{*2}\subseteq D_E^*(\varphi)^2$, and thus the inclusion $D_E^*(\psi)^2\subseteq D_E^*(\varphi)^2$ holds for trivial reasons.

Suppose $\dim\psi'\geq2$; then $\psi'\subseteq\psi$ implies that $\psi_{F(\psi')}$ is isotropic. Together with the assumptions that $\varphi_{F(\psi)}$ is isotropic and $\psi$ is nondefective, we get that $\varphi_{F(\psi')}$ is isotropic by Proposition~\ref{Prop:TransitionOfIsotropy}.

\smallskip

(3) Now fix a field $E\in\mE$. By (1) we can assume that $\varphi,\psi$ are anisotropic over $F$ and $\dim\varphi,\dim\psi\geq2$. Invoking (2) we can suppose (by replacing $\psi$ with $\psi'$) that $\psi_E$ is nondefective.  We treat different kinds of quadratic forms separately.

\smallskip

First assume that $\varphi_E$ is anisotropic. If $\psi_E$ is isotropic, then (since we assume $\psi_E$ to be nondefective), the extension $E(\psi)/E$ is purely transcendental by Lemma~\ref{Lemma:FunFieldTransc}, and so, by Lemma~\ref{Lemma:IsotropyTransc}, the isotropy of $\varphi_{E(\psi)}$ implies the isotropy of $\varphi_E$, a contradiction. Therefore, $\psi_E$ must be anisotropic in this case, and the claim follows from Lemma~\ref{Lemma:MainPart-FormsAnisOverE}.

\smallskip

Now suppose $\varphi_E$ is isotropic. If $\iw{\varphi_E}>0$, then we are done by (0). Hence, assume $\iti{\varphi_E}=\iql{\varphi_E}>0$. Note that $\iql{\varphi_{F(\psi)}}>0$ means that $\ql{\varphi}_{F(\psi)}$ is isotropic, which is possible only if $\psi$ is quasilinear by Lemma~\ref{Lemma:IsotropySQFoverTSQF}; but in that case we have $\iql{\varphi_E}=0$ by the assumption. Therefore, it must be   $\iql{\varphi_{F(\psi)}}=0$, and hence $\iw{\varphi_{F(\psi)}}>0$. If $\psi_E$ were isotropic, then (since $\psi_E$ is nondefective) $E(\psi)/E$ were a purely transcendental extension by Lemma~\ref{Lemma:FunFieldTransc}, and hence, by Lemma~\ref{Lemma:IsotropyTransc},
\[0=\iw{\varphi_E}=\iw{\varphi_{E(\psi)}}\geq\iw{\varphi_{F(\psi)}}>0,\]
a contradiction; therefore, $\psi_E$ must be anisotropic. Moreover, note that the assumptions $\iql{\varphi_E}>0$ and $\iw{\varphi_{F(\psi)}}>0$ imply that $\varphi$ is semisingular. Set $\varphi'=(\varphi_E)_{\an}$; by Lemma~\ref{Lemma:TechnicalIsotropyOverFunField}, we have $\iw{\varphi'_{E(\psi)}}>0$. Applying Lemma~\ref{Lemma:MainPart-FormsAnisOverE} to the forms $\varphi'$ and $\psi$, we get $D_E^*(\psi)^2\subseteq D_E^*(\varphi')^2$. Since $D_E^*(\varphi')=D_E^*(\varphi)$, the claim follows.
\end{proof}

Combining Proposition~\ref{Prop:RousAn_NecCond}, Lemma~\ref{Lemma:MultiplicativeGroups} and Lemma~\ref{Lemma:RousAn_SuffCond} applied to the case of quadratic forms, we get a full characterisation of the isotropy of $\varphi_{F(\psi)}$.

\begin{theorem}
\label{Th:CharacterizationIsotropyOverFunField}
Let $\varphi$, $\psi$ be nondefective quadratic forms over $F$, and suppose that $\dim\psi\geq2$. 
Denote $\bX=(X_1,\dots, X_{\dim\psi})$, and set
\[
\mE=\begin{cases}
\{E~|~E/F \text{ an extension s.t. } \iql{\varphi_E}=0\} & \text{ if } \psi \text{ quasilinear,}\\
\{E~|~E/F \text{ an extension}\} & \text{ otherwise.}
\end{cases}
\]
Then the following assertions are equivalent:
\begin{enumerate}
	\item $\varphi_{F(\psi)}$ is isotropic, 
	\item $D_E^*(\psi)^2\subseteq D_E^*(\varphi)^2$ for every $E\in\mE$,
	\item $a\psi(\bX)\in D_{F(\bX)}^*(\varphi)^2$ for every $a\in D_F^*(\psi)$,
	\item $\Ng{E}{\psi}\subseteq \Ng{E}{\varphi}$ for every $E\in\mE$,
	\item $a\psi(\bX)\in \Ng{F(\bX)}{\varphi}$ for every $a\in D_F^*(\psi)$,
	\item $\Tg{E}{a\psi}\subseteq \Tg{E}{\varphi}$ for every $E\in\mE$ and every $a\in D_F^*(\psi)$, 
	\item $a\psi(\bX)\in \Tg{F(\bX)}{\varphi}$ for every $a\in D_F^*(\psi)$.
\end{enumerate}
\end{theorem}

\begin{proof}
(i)~$\Rightarrow$~(ii) is covered by Proposition~\ref{Prop:RousAn_NecCond}. The implications (ii)~$\Rightarrow$~(iii), (iv)~$\Rightarrow$~(v) and (vi)~$\Rightarrow$~(vii) follow trivially by setting $E=F(\bX)$; note that $F(\bX)\in\mE$ by Lemma~\ref{Lemma:IsotropyTransc}. By Lemma~\ref{Lemma:MultiplicativeGroups}, we get (iii)~$\Rightarrow$~(v)~$\Rightarrow$~(vii).

To prove (ii)~$\Rightarrow$~(iv), let $x\in \Ng{E}{\psi}$. Then $x=\prod_{i=1}^ka_ib_i$ for some $a_i,b_i\in D_E^*(\psi)$. Since $a_ib_i\in D_E^*(\varphi)^2$ by the assumption of (ii), we get $x\in\Ng{E}{\psi}$.

By combining the assumption of (iv) with Lemma~\ref{Lemma:MultiplicativeGroups}, we obtain that $\Tg{E}{a\psi}=\Ng{E}{\psi}\subseteq\Ng{E}{\varphi}\subseteq\Tg{E}{\varphi}$; therefore, (iv)~$\Rightarrow$~(vi).

Finally, Lemma~\ref{Lemma:RousAn_SuffCond} gives (vii)~$\Rightarrow$~(i) for all but one possible $\psi$; if $\psi\cong\H$, then the polynomial $\psi(\bX)$ is reducible. But by assuming (vii), we have $X_1X_2\in\Tg{F(X_1,X_2)}{\varphi}$, and it follows by Theorem~\ref{Th:EBF} that $\varphi_{F}$ must be isotropic (because for $f_i=X_i$, we have $F(f_i)\simeq F$). Since $F(\H)=F$ by definition, the claim follows.
\end{proof}

\begin{remark}
The assumption that $\dim\psi\geq2$ is only necessary for the proof of (vii) $\Rightarrow$ (i), but it is crucial there: Let $\psi\cong\sqf{1}$ and $\varphi\cong\nsqf{1,b}$ for some $b\in F\setminus\wp(F)$; then $\varphi$ is anisotropic over $F$. We have $a^2X_1^2\cdot1\in \Tg{F(X_1)}{\varphi}$ for any $a\in F^*$, i.e., (vii) is fulfilled. But $F(\psi)=F$, so (i) does not hold.  
\end{remark}

In the proof of Theorem~\ref{Th:CharacterizationIsotropyOverFunField}, to show (ii) $\Rightarrow$ (i), we used the chain of implications (ii) $\Rightarrow$ (iii) $\Rightarrow$ (v) $\Rightarrow$ (vii) $\Rightarrow$ (i), in which we needed only the field $F(\bX)$. Moreover, we could as well use any field $F(X_1,\dots,X_n)$ with $n\geq\dim\psi$. Hence, we have actually proved the following:

\begin{corollary} \label{Cor:CharacterizationIsotropyOverFunField_Simplified}
Let $\varphi,\psi$ be nondefective quadratic forms with $\dim\psi\geq2$, and let $\bY=(Y_1,\dots,Y_n)$ with $n\geq\dim\psi$. Then the following are equivalent:
\begin{enumerate}
	\item $\varphi_{F(\psi)}$ is isotropic,
	\item $D_{F(\bY)}^*(\psi)^2\subseteq D_{F(\bY)}^*(\varphi)^2$.
\end{enumerate}
\end{corollary}

Let $\varphi$, $\psi$ be quadratic forms over $F$. We can ask when both $\varphi_{F(\psi)}$ and $\psi_{F(\varphi)}$ are isotropic, i.e., when
\begin{equation}\label{Eq:StbCond}\tag{$\ast$}
\iti{\varphi_{F(\psi)}}>0 \quad \& \quad \iti{\psi_{F(\varphi)}}>0.
\end{equation}
If the forms $\varphi$ and $\psi$ are nondefective, then \eqref{Eq:StbCond} is an equivalent condition to $\varphi$ and $\psi$ being \emph{stably birationally equivalent}, which we denote $\varphi\simstb\psi$. Theorem~\ref{Th:CharacterizationIsotropyOverFunField} together with Corollary~\ref{Cor:CharacterizationIsotropyOverFunField_Simplified} give a characterization of this phenomenon.

\begin{corollary} \label{Cor:CharSTB}
Let $\varphi$ and $\psi$ be nondefective quadratic forms of dimension at least two, and let $\bY=(Y_1,\dots,Y_n)$ with $n\geq\max\{\dim\varphi, \dim\psi\}$. Moreover, set 
\[
\mE=\begin{cases}
\{E~|~E/F \text{ an extension}\} \quad \text{ if neither $\varphi$ nor $\psi$ is quasilinear,}\\
\{E~|~E/F \text{ an extension s.t. } \iql{\varphi_E}=\iql{\psi_E}=0\}  \quad \text{ otherwise.}
\end{cases}
\]

Then the following are equivalent:
\begin{enumerate}[(a)]
	\item $\varphi\simstb\psi$,
	\item $D_E^*(\psi)^2= D_E^*(\varphi)^2$ for every $E\in\mE$,
	\item $\Ng{E}{\psi}= \Ng{E}{\varphi}$ for every $E\in\mE$,
	\item $D_{F(\bY)}^*(\psi)^2=D_{F(\bY)}^*(\varphi)^2$.
\end{enumerate}
If, moreover, $1\in D_F^*(\varphi)\cap D_F^*(\psi)$, then the above are also equivalent to
\begin{enumerate}[(a)] \setcounter{enumi}{4}
	\item $\Tg{E}{\psi}= \Tg{E}{\varphi}$ for every $E\in\mE$.
\end{enumerate}
\end{corollary}

\begin{proof}
The proof can basically be obtained through two-sided applications of Theorem~\ref{Th:CharacterizationIsotropyOverFunField} and Corollary~\ref{Cor:CharacterizationIsotropyOverFunField_Simplified}, but note that our current $\mE$ is slightly different from the one in that theorem.

The implication (a) $\Rightarrow$ (b) follows directly from the theorem, and (b) $\Rightarrow$ (c) can be done exactly as the proof of (ii) $\Rightarrow$ (iv) of the theorem. For (c) $\Rightarrow$ (a), note that to prove (iv) $\Rightarrow$ (i) in the theorem, we have  actually shown (iv) $\Rightarrow$ (v) $\Rightarrow$ (vii) $\Rightarrow$ (i), and we only used the field $F(\bX)$, which is an element of (our current) $\mE$; therefore, (c) $\Rightarrow$ (a) holds.

The implication (b) $\Rightarrow$ (d) is obvious, as $F(\bY)\in\mE$. Furthermore, (d) $\Rightarrow$ (a) follows directly from the corollary.

Finally, (a) $\Leftrightarrow$ (e) follows from the theorem if we note that in the proof of (vi) $\Rightarrow$ (vii) $\Rightarrow$ (i), only the existence of one $a\in D_F^*(\psi)$ with the required property was necessary.
\end{proof}

\section{Sums and Pfister multiples of quadratic forms} \label{Sec:Pfister}

The goal of this section is to extend the results from the previous one by comparing the isotropy of $\varphi_{F(\psi)}$ with the isotropy of $(\pi\otimes\varphi)_{F(\pi\otimes\psi)}$ for a bilinear Pfister form $\pi$.

\begin{lemma} \label{Lemma:AnisotropyPurelyInsepAndFunField}
Let $\varphi_0$, $\varphi_1$ and $\tau$ be quadratic forms over $F$ with $\varphi_0$ and $\varphi_1$ nondefective, and let $\varphi=\varphi_0\ort X \varphi_1$ over $F(X)$. Then $\iw{\tau_{F(X)(\varphi)}}=\iw{\tau}$ and $\iql{\tau_{F(X)(\varphi)}}=\iql{\tau}$. In particular, if $\tau$ is nondefective, then $\tau_{F(X)(\varphi)}$ is also nondefective.
\end{lemma}

\begin{proof}
If $\varphi_i$ is isotropic for some $i$, then necessarily $\H\subseteq\varphi$, and hence $F(X)(\varphi)/F$ is purely transcendental by Lemma~\ref{Lemma:FunFieldTransc} (because $\varphi$ is nondefective by Example~\ref{Ex:PowSerAnisotropy}). Then the claim follows by Lemma~\ref{Lemma:IsotropyTransc}. Hence, we can suppose that $\varphi_0$ and $\varphi_1$ are anisotropic.

Assume that $\tau$ is anisotropic. Let $\dim\tau=d$, $\dim\varphi=n$, and denote $\bY=(Y_1,\dots,Y_n)$, an $n$-tuple of variables. Without loss of generality, suppose that $1\in D_{F(X)}^*(\varphi)$. For a contradiction, assume that $\tau_{F(X)(\varphi)}$ is isotropic; then, by Theorem~\ref{Th:CharacterizationIsotropyOverFunField}, $\varphi(\bY)\in D_{F(X,\bY)}^*(\tau)^2$. Thus, after the usual multiplication by a common denominator, we get
\begin{equation} \label{Eq:ComparingXvaluation}
h^2(X,\bY)\varphi(\bY)=\prod_{i=1}^2\tau(\xi_{i1}(X,\bY),\dots,\xi_{id}(X,\bY))
\end{equation}
for some $h\in F[X,\bY]$ and $\xi_{ij}\in F[X,\bY]$, $1\leq i\leq 2$, $1\leq j \leq d$. Clearly, $\deg_X(\tau(\xi_{i1}(X,\bY),\dots,\xi_{id}(X,\bY)))$ is even for each $i$, and hence the degree in $X$ of the polynomial on the right side of \eqref{Eq:ComparingXvaluation} is even. On the other hand, we know $\deg_X(\varphi(\bY))=1$, so the degree in $X$ of the polynomial on the left side of \eqref{Eq:ComparingXvaluation} is odd. That is absurd; therefore, $\tau_{F(X)(\varphi)}$ is anisotropic.

If $\tau$ is isotropic, then we know by the previous part of the proof that $\tau_{\an}$ remains anisotropic over $F(X)(\varphi)$. The claim follows.
\end{proof}

\begin{remark}
As it was pointed out by the anonymous referee, the proof of Lemma~\ref{Lemma:AnisotropyPurelyInsepAndFunField} can be simplified: the field $F(X)(\varphi)$ is rational over $F$, because the defining relation $\varphi_0(\bY_0)+X\varphi_1(\bY_1)=0$ can be used to eliminate $X$. However, I decided to keep the original proof as a demonstration of the application of Theorem~\ref{Th:CharacterizationIsotropyOverFunField}.
\end{remark}

\begin{proposition}
 \label{Prop:XsumsOfQFs}
Let $\varphi_0$, $\varphi_1$, $\psi_0$, $\psi_1$ be nondefective quadratic forms over $F$. Write $F'=F(X)$ and $F''=F\dbrac{X}$. Let $\varphi=\varphi_0\ort X\varphi_1$, $\psi=\psi_0\ort X\psi_1$ be quadratic forms over $F'$, and set
\[\mE=\{E~|~E/F \text{ s.t. } \iql{(\varphi_0)_E}=\iql{(\varphi_1)_E}=\iql{(\psi_0)_E}=\iql{(\psi_1)_E}=0\}.\]
\begin{enumerate}[(1)]
	\item The following are equivalent:
				\begin{enumerate}[(i), font=\normalfont]
						\item $\varphi_{F'(\psi)}$ is isotropic,
						\item $\varphi_{F''(\psi)}$ is isotropic,
						\item $D_E^*(\psi_0)D_E^*(\psi_1)\subseteq D_E^*(\varphi_0)D_E^*(\varphi_1)$ for each $E\in\mE$.
				\end{enumerate}
	\item Let the form $\varphi_{F'(\psi)}$ (or the form $\varphi_{F''(\psi)}$) be isotropic. If $\dim\psi_i\geq2$ for an $i\in\{0,1\}$, then at least one of the forms $\varphi_0$ and $\varphi_1$ is isotropic over $F(\psi_i)$.
\end{enumerate}
\end{proposition}

\begin{proof}
We start by proving (1). The implication (i) $\Rightarrow$ (ii) is obvious. 

For (ii) $\Rightarrow$ (iii), let $E\in\mE$, and denote $E'=E(X)$ and $E''=E\dbrac{X}$. If $(\varphi_i)_E$ is isotropic for some $i$, then, since $\iql{(\varphi_i)_E}=0$ by the assumption, we have $\H\subseteq\varphi_i$; in that case, $D_E^*(\varphi_0)D_E^*(\varphi_1)=E^*$, and the claim follows trivially. Thus, assume that both $\varphi_0$ and $\varphi_1$ are anisotropic over $E$. Let $a\in D_E^*(\psi_0)D_E^*(\psi_1)$; then $(\psi_0\ort a\psi_1)_E$ is isotropic. Obviously, we have $\psi_{E''(\sqrt{aX})}\cong(\psi_0\ort a\psi_1)_{E''(\sqrt{aX})}$, and hence $\psi_{E''(\sqrt{aX})}$ is isotropic. As $\varphi_{F''(\psi)}$ is isotropic by the assumption, $\varphi_{E''(\psi)}$ must be isotropic as well. Furthermore, $\varphi_{E''}$ and $\psi_{E''}$ are nondefective (see Example~\ref{Ex:PowSerAnisotropy}). It follows from Proposition~\ref{Prop:TransitionOfIsotropy} that $\varphi_{E''(\sqrt{aX})}$ is isotropic. Moreover, we have $\varphi_{E''(\sqrt{aX})}\cong(\varphi_0\ort a\varphi_1)_{E''(\sqrt{aX})}$. Since $E''(\sqrt{aX})$ is a complete discrete valuation field with residue field $E$, the form $(\varphi_0\ort a\varphi_1)_E$ is isotropic by Lemma~\ref{Lemma:CDVfieldsIsotropy}. Since both $\varphi_0$ and $\varphi_1$ are anisotropic over $E$, we get that $a\in D_E^*(\varphi_0)D_E^*(\varphi_1)$.

Now we prove (iii) $\Rightarrow$ (i): First, note that $F'(\psi)\in\mE$ by Lemma~\ref{Lemma:AnisotropyPurelyInsepAndFunField}. Since $\psi_{F'(\psi)}\cong(\psi_0\ort X\psi_1)_{F'(\psi)}$ is isotropic, we get $X\in D_{F'(\psi)}^*(\psi_0)D_{F'(\psi)}^*(\psi_1)$. Therefore, $X\in D_{F'(\psi)}^*(\varphi_0)D_{F'(\psi)}^*(\varphi_1)$, which implies that the quadratic form $(\varphi_0\ort X\varphi_1)_{F'(\psi)}\cong\varphi_{F'(\psi)}$ is isotropic.

\smallskip

To prove (2), note that it follows from the assumptions that $\varphi_{F''(\psi)}$ is isotropic in any case. Observe that if $\psi_0$ or $\psi_1$ is isotropic, then $\varphi$ must be isotropic over $F''$ by Lemma~\ref{Lemma:IsotropyTransc}, and in that case $\varphi_0$ or $\varphi_1$ is isotropic over $F$ by Lemma~\ref{Lemma:CDVfieldsAnisotropy}. But if $\varphi_0$ or $\varphi_1$ is isotropic over $F$, then the claim is trivial. Hence, suppose that $\varphi_0$, $\varphi_1$, $\psi_0$, $\psi_1$ are all anisotropic. 

Assume that $i\in\{0,1\}$ such that $\dim\psi_i\geq2$; then $\psi_{F''(\psi_i)}$ is isotropic. As $\varphi_{F''}$, $\psi_{F''}$ and $(\psi_i)_{F''}$ are anisotropic, it follows by Proposition~\ref{Prop:TransitionOfIsotropy} that $\varphi_{F''(\psi_i)}$ is isotropic, too. Since $F''(\psi_i)\subseteq F(\psi_i)\dbrac{X}$, it follows that $\varphi$ is isotropic over $F(\psi_i)\dbrac{X}$. By Lemma~\ref{Lemma:CDVfieldsAnisotropy}, $\varphi_0$ or $\varphi_1$ must be isotropic over $F(\psi_i)$.
\end{proof}

\begin{corollary}
Let $\varphi_0$, $\varphi_1$, $\psi_0$, $\psi_1$, $\varphi$, $\psi$, $\mE$ be as in Proposition~\ref{Prop:XsumsOfQFs}. Then the following are equivalent:
\begin{enumerate}
	\item $\varphi\simstb\psi$,
	\item $D_E^*(\psi_0)D_E^*(\psi_1)= D_E^*(\varphi_0)D_E^*(\varphi_1)$ for each $E\in\mE$.
\end{enumerate}
\end{corollary}

An \emph{$n$-fold bilinear Pfister form} is a bilinear form $\sqf{1,a_1}_\b\otimes\cdots\otimes\sqf{1,a_n}_\b$ with $a_1,\dots,a_n\in F^*$; we write $\pf{a_1,\dots,a_n}_\b$ for short.

\begin{lemma}
\label{Lemma:XXXPFmultiples}
Let $\varphi$, $\psi$ be nondefective quadratic forms over $F$, $\dim\psi\geq2$, and let
\[\mE=\{E~|~E/F \text{ an extension s.t. } \iql{\varphi_E}=\iql{\psi_E}=0\}.\]
Then the following are equivalent:
\begin{enumerate}
	\item $\varphi_{F(\psi)}$ is isotropic,
	\item $D_E(\psi)^2\subseteq D_E(\varphi)^2$ for each $E\in\mE$,
	\item $\varphi\ort X\varphi$ is isotropic over $F(X)(\psi\ort X\psi)$,
	\item $\varphi\ort X\varphi$ is isotropic over $F\dbrac{X}(\psi\ort X\psi)$,
	\item for every $n\geq0$, $(\pf{X_1,\dots,X_n}_{\b}\otimes \varphi)_{F(X_1,\dots,X_n)(\pf{X_1,\dots,X_n}_{\b}\otimes\psi)}$ is isotropic,
	\item for every $n\geq0$, $(\pf{X_1,\dots,X_n}_{\b}\otimes \varphi)_{F\dbrac{X_1}\dots\dbrac{X_n}(\pf{X_1,\dots,X_n}_{\b}\otimes\psi)}$ is isotropic.
\end{enumerate}
\end{lemma}

\begin{proof}
(i) $\Rightarrow$ (ii) follows from Theorem~\ref{Th:CharacterizationIsotropyOverFunField}, and, since $F(X_1,\dots,X_{\dim\psi})\in\mE$ by Lemma~\ref{Lemma:IsotropyTransc}, (ii) $\Rightarrow$ (i) follows from Corollary~\ref{Cor:CharacterizationIsotropyOverFunField_Simplified}. Equivalences (ii) $\Leftrightarrow$ (iii) $\Leftrightarrow$ (iv) are a consequence of Proposition~\ref{Prop:XsumsOfQFs}. Furthermore, (iii) is a special case of (v). For the other way around, note that (v) coincides with (i) for $n=0$; for $n=1$, (v) coincides with (iii); and for $n>1$, (v) can be obtained by a repeated application of (i) $\Rightarrow$ (iii). Thus, (iii) $\Leftrightarrow$ (v). Analogously, we get (iv) $\Leftrightarrow$ (vi).
\end{proof}

\begin{theorem}\label{Th:IsotropyPfisterMultiples} 
Let $\varphi$, $\psi$ be quadratic forms over $F$, and let $\pi$ be a bilinear Pfister form over $F$. If $\varphi_{F(\psi)}$ is isotropic, then $(\pi\otimes\varphi)_{F(\pi\otimes\psi)}$ is isotropic.
\end{theorem}

\begin{proof}
Suppose without loss of generality that $\varphi$, $\psi$, $\pi$ are anisotropic over $F$. Let $\pi\cong\pf{a_1,\dots,a_n}_{\b}$. If $n=0$, i.e., $\pi\cong\sqf{1}$, then there is nothing to prove; thus, assume $n\geq1$.

Since $\varphi_{F(\psi)}$ is isotropic, we get by Lemma~\ref{Lemma:XXXPFmultiples} that the quadratic form $\pf{X_1,\dots,X_n}_{\b}\otimes\varphi$ is isotropic over $F(X_1, \dots, X_n)(\pf{X_1,\dots,X_n}_{\b}\otimes\psi)$.  Set $T=X_n+a_n$; then we have $F(X_1, \dots, X_n)=F(X_1, \dots, X_{n-1},T)$, and it follows that $F(X_1, \dots, X_{n-1},T)\subseteq F(X_1, \dots, X_{n-1})\dbrac{T}$. Let us write $K=F(X_1, \dots, X_{n-1})\dbrac{T}(\sqrt{a_nT})$; then $1+\frac{T}{a_n}\in K^{*2}$, and so
\[X_n=a_n+T\equiv (a_n+T)\left(1+\frac{T}{a_n}\right)=\frac{(a_n+T)^2}{a_n}\equiv\frac{1}{a_n}\equiv a_n \quad \mod K^{*2}.\]
Therefore, over $K$, we have
\begin{align*}
(\pf{X_1,\dots,X_n}_{\b}\otimes\varphi)_K&\cong(\pf{X_1,\dots,X_{n-1}, a_n}_{\b}\otimes\varphi)_K,\\
(\pf{X_1,\dots,X_n}_{\b}\otimes\psi)_K&\cong(\pf{X_1,\dots,X_{n-1}, a_n}_{\b}\otimes\psi)_K.
\end{align*}
Write $\varphi'\cong\pf{X_1,\dots,X_{n-1}, a_n}_{\b}\otimes\varphi$ and $\psi'\cong\pf{X_1,\dots,X_{n-1}, a_n}_{\b}\otimes\psi$; we have $\varphi_{K(\psi)}\cong\varphi'_{K(\psi')}$, and it follows that $\varphi'_{K(\psi')}$ is isotropic. Note that the forms $\varphi'$, $\psi'$ are defined over $F(X_1,\dots, X_{n-1})$, so in particular, we have $K(\psi')\subseteq F(X_1,\dots, X_{n-1})(\psi')\dbrac{T}(\sqrt{a_nT})$. Hence, $\varphi'$ is isotropic over $F(X_1,\dots, X_{n-1})(\psi')\dbrac{T}(\sqrt{a_nT})$, which is a complete discrete valuation field with the residue field $F(X_1,\dots,X_n)(\psi')$. Therefore, by Lemma~\ref{Lemma:CDVfieldsIsotropy}, $\varphi'_{F(X_1,\dots,X_{n-1})(\psi')}$ is isotropic. We proceed by induction.
\end{proof}

\begin{corollary} \label{Cor:CharIsoThroughPF}
Let $\varphi$, $\psi$ be quadratic forms over $F$. Then the following are equivalent:
\begin{enumerate}
	\item $\varphi_{F(\psi)}$ is isotropic,
	\item for any field extension $E/F$, any $n\geq 0$ and any $n$-fold bilinear Pfister form $\pi$ over $E$, the form $(\pi\otimes\varphi)_{E(\pi\otimes\psi)}$ is isotropic.
\end{enumerate}
\end{corollary}

\begin{remark}
One could hope that, given anisotropic quadratic forms $\varphi$, $\psi$ and a bilinear Pfister form $\pi$ over $F$ such that $(\pi\otimes\varphi)_{F(\pi\otimes\psi)}$ is isotropic, then $\varphi_{F(\psi)}$ must be isotropic. But that is not true in general: Let $a,b\in F$ be $2$-independent (i.e., the set $\{1, a, b, ab\}$ is linearly independent over $F^2$), and set $\varphi\cong\sqf{1,a}$, $\psi\cong\sqf{1,a,b}$, and $\pi\cong\pf{a}_{\b}$. Then $\varphi_{F(\psi)}$ is anisotropic by \cite[Theorem~1.1]{HL06}, while $\pi\otimes\varphi$ is obviously isotropic over $F$ already.

Another example shows that the claim cannot hold even under the assumption that $\pi\otimes\varphi$ is anisotropic over $F$: Let $a,b$ be as above, and set $\varphi'\cong\sqf{1,a}$, $\psi'\cong\sqf{1,ab}$ and $\pi'\cong\pf{b}_{\b}$. Then $\pi'\otimes\varphi'\cong\sqf{1,a,b,ab}\cong\pi'\otimes\psi'$, so $(\pi'\otimes\varphi')_{F(\pi'\otimes\psi')}$ is obviously isotropic. On the other hand, if $\varphi'_{F(\psi')}$ were isotropic, then, by Proposition~\ref{Prop:IsotropyOverQuadrExt}, it must hold $\varphi'\simsim\psi'$,  which is not true. 
\end{remark}

As the final result, we combine Corollary~\ref{Cor:CharSTB}, Lemma~\ref{Lemma:XXXPFmultiples} and Corollary~\ref{Cor:CharIsoThroughPF}.

\begin{corollary}\label{Cor:CharSTB-PF}
Let $\varphi$, $\psi$ be nondefective quadratic forms over $F$ of dimension at least two, and let
\[\mE=\{E~|~E/F \text{ an extension s.t. } \iql{\varphi_E}=\iql{\psi_E}=0\}.\] 
Then the following are equivalent:
\begin{enumerate}
	\item $\varphi\simstb\psi$,
	\item $D_E(\psi)^2= D_E(\varphi)^2$ for each $E\in\mE$,
	\item for every $n\geq1$, $\pf{X_1,\dots,X_n}_{\b}\otimes \varphi\simstb \pf{X_1,\dots,X_n}_{\b}\otimes \psi$ over $F(X_1,\dots,X_n)$,
	\item for any field extension $E/F$, any $n\geq 0$ and any $n$-fold bilinear Pfister form $\pi$ over $E$, we have $(\pi\otimes\varphi)_E\simstb(\pi\otimes\psi)_E$.
\end{enumerate}
\end{corollary}

\newpage
\printbibliography

\end{document}